\documentclass[reqno]{amsart}
\usepackage{amssymb}
\usepackage{amsmath}
\usepackage{amsfonts}

\newtheorem{theorem}{Theorem}[section]
\theoremstyle{plain}

\newtheorem{definition}{Definition}[section]

\newtheorem{lemma}{Lemma}[section]

\newtheorem{remark}{Remark}[section]

\numberwithin{equation}{section}

\begin{document}
\title[Attractors for the Strongly Damped Wave Equation]{Attractors for the
Strongly Damped Wave Equation with $p$-Laplacian}
\author{Azer Khanmamedov \ }
\address{{\small Department of Mathematics,} {\small Faculty of Science,
Hacettepe University, Beytepe 06800}, {\small Ankara, Turkey}}
\email{azer@hacettepe.edu.tr}
\author{\ Zehra \c{S}en}
\address{{\small Department of Mathematics,} {\small Faculty of Science,
Hacettepe University, Beytepe 06800}, {\small Ankara, Turkey}}
\email{zarat@hacettepe.edu.tr}

\begin{abstract}
This paper is concerned with the initial boundary value problem for one
dimensional strongly damped wave equation involving $p$-Laplacian. For $p>2$%
, we establish the existence of weak local attractors for this problem in $%
W_{0}^{1,p}(0,1)\times L^{2}(0,1)$. Under restriction $2<p<4$, we prove that
the semigroup, generated by the considered problem, possesses a strong
global attractor in $W_{0}^{1,p}(0,1)\times L^{2}(0,1)$ and this attractor
is\ a bounded subset of $W^{1,\infty }(0,1)\times W^{1,\infty }(0,1)$.
\end{abstract}

\maketitle

\section{Introduction}

In this paper, we consider the following strongly damped wave equation with $%
p$-Laplacian:%
\begin{equation}
u_{tt}-u_{txx}-\frac{\partial }{\partial x}\left( \left\vert
u_{x}\right\vert ^{p-2}u_{x}\right) +f\left( u\right) =g.  \tag{1.1}
\end{equation}%
The strongly damped wave equations occur in so many physical areas, such as
heat conduction, solid mechanics and so on, which has considerably attracted
many authors to analyze, in particular, the long time dynamics of these
types of equations. It is well known that the long time dynamics of
evolution equations can be described in terms of attractors. The attractors
for (1.1), with $p=2$, have been widely studied by several authors in
multidimensional case under different hypothesis. We refer to [1-11] for
wave equations with the linear strong damping, and to [12-13] for wave
equations with the nonlinear strong damping. These works, as has been
mentioned above, deal with attractors of strongly damped wave equations
involving linear Laplacian. For the strongly damped wave equation with
nonlinear Laplacian, we refer to \cite{14} and \cite{15}. In \cite{14}, the
authors studied the long time behaviour of regular, precisely from the space
$H^{2}$, solutions of the strongly damped wave equation involving nonlinear
Laplacian in the form $\frac{\partial }{\partial x}\sigma (u_{x})$, with $%
\sigma \in C^{1}(%
\mathbb{R}
)$ and $\sigma ^{\prime }(\cdot )\geq r_{0}>0$. Because, under these
conditions on $\sigma $, the nonlinear term $\frac{\partial }{\partial x}%
\sigma (u_{x})$ behaves like $u_{xx}$ for the $H^{2}$-solutions, the authors
were able to successfully apply the splitting method to prove the asymptotic
compactness of the solutions. Although, in \cite{15}, the authors studied
the attractors (in the weak topology) for the weak solutions of strongly
damped wave equation with more general nonlinear Laplacian, the equation
considered in that article contains the additional\ term $-\Delta u$, in
comparison with (1.1). This term, together with nonlinear degenerate
Laplacian, generates indeed non-degenerate Laplacian and thereby allows to
obtain some additional estimates for weak and strong solutions. Unlike the
equations considered in \cite{14} and \cite{15}, the equation (1.1) involves
degenerate Laplacian and therefore we are not able to apply the approaches
of those papers, especially in the study of the strong global attractor. On
the other hand, there are some difficulties also in application of the
asymptotic compactness methods developed in [16-18]. The difficulties are
caused by the absence of the energy equality for a weak solution and the
energy inequality for the difference of two weak solutions of (1.1). To
overcome these difficulties, we require an additional restriction on the
exponent $p$. Namely, under restriction \ $2<p<4$, by using specificity of
the one dimensional case, we show that weak local attractors, the existence
of which we establish for every $p>2$, are bounded subsets of $W^{1,\infty
}(0,1)\times $ $W^{1,\infty }(0,1)$. This fact implies the validity of the
energy equality for the trajectories from the weak local attractors, and
applying the methods of [16-18], we establish the asymptotic compactness of
the weak solutions, which, together with\ the presence of strict Lyapunov
function, leads to the existence of a strong global attractor in $%
W_{0}^{1,p}(0,1)\times L^{2}(0,1)$.

The paper is organized as follows. In the next section, we give the
statement of the problem and main results. In Section 3, we prove the
existence of the weak local attractors in\ $W_{0}^{1,p}(0,1)\times
L^{2}(0,1) $ and show that these attractors attract trajectories in the
topology of $H_{0}^{1}(0,1)\times L^{2}(0,1)$. In Section 4, we first prove
the regularity of the weak local attractors, for $p\in (2,4)$, and then we
establish the existence of the strong global attractor in $%
W_{0}^{1,p}(0,1)\times L^{2}(0,1)$. Finally, in the last section, we give
some auxiliary lemmas.

\section{Statement of the problem and\ main results}

We deal with the following initial-boundary value problem:
\begin{equation}
\left\{
\begin{array}{l}
u_{tt}-u_{txx}-\frac{\partial }{\partial x}\left( \left\vert
u_{x}\right\vert ^{p-2}u_{x}\right) +f\left( u\right) =g\left( x\right) ,%
\text{ \ in }\left( 0,\infty \right) \times \left( 0,1\right) , \\
u\left( \cdot ,0\right) =u\left( \cdot ,1\right) =0,\text{ in }\left(
0,\infty \right) ,\text{ } \\
u\left( 0,\cdot \right) =u_{0}\left( \cdot \right) \ ,\ \ \ u_{t}\left(
0,\cdot \right) =u_{1}\left( \cdot \right) ,\text{ \ in }(0,1).%
\end{array}%
\right.  \tag{2.1}
\end{equation}%
Here
\begin{equation}
p>2,\text{ \ }g\in L^{2}(0,1),  \tag{2.2}
\end{equation}%
and the function $f$ satisfies the following conditions:%
\begin{equation}
f\in C^{1}\left(
\mathbb{R}
\right) \text{ \ , \ }\liminf_{\left\vert s\right\vert \rightarrow \infty }%
\frac{f\left( s\right) }{\left\vert s\right\vert ^{p-2}s}>-\lambda ^{p},
\tag{2.3}
\end{equation}%
where $\lambda =\inf\limits_{\varphi \in W_{0}^{1,p}(0,1),\varphi \neq 0}%
\frac{\left\Vert \varphi ^{\prime }\right\Vert _{L^{p}(0,1)}}{\left\Vert
\varphi \right\Vert _{L^{p}(0,1)}}.$

\bigskip

Let us recall the following definitions.

\begin{definition}
The function $u\in L^{1}(0,T;W_{0}^{1,p}(0,1))$ satisfying $u_{t}\in
L^{1}(0,T;W_{0}^{1,\frac{p}{p-1}}(0,1))\cap C([0,T];W^{-1,\frac{p}{p-1}%
}(0,1))$, $u(0,x)=u_{0}(x)$, $u_{t}(0,x)=u_{1}(x)$ and the equation%
\begin{equation*}
\frac{d}{dt}\int\limits_{0}^{1}u_{t}(t,x)v(x)dx+\int%
\limits_{0}^{1}u_{tx}(t,x)v^{\prime }(x)dx+\int\limits_{0}^{1}\left\vert
u_{x}(t,x)\right\vert ^{p-2}u_{x}(t,x)v^{\prime }(x)dx
\end{equation*}%
\begin{equation*}
+\int\limits_{0}^{1}f(u(t,x))v(x)dx=\int\limits_{0}^{1}g(x)v(x)dx,
\end{equation*}%
in the sense of distributions on $(0,T)$, for all $v\in W_{0}^{1,p}(0,1)$,
is called a weak solution to the problem (2.1) in $\left[ 0,T\right] \times
\lbrack 0,1]$.
\end{definition}

\begin{definition}
Let $\{V(t)\}_{{\small t\geq 0}}$ be an operator semigroup on a linear
normed space $E$ and $B$ be a bounded subset of $E$. A set $\mathcal{A}%
_{B}\subset E$ is called a strong (weak) local attractor for $B$ and the
semigroup $\left\{ V(t)\right\} _{t\geq 0}$ iff

$\bullet $ $\mathcal{A}_{B}$ is strongly (weakly) compact in $E$;

$\bullet $ $\mathcal{A}_{B}$ is invariant, i.e. $V(t)\mathcal{A}_{B}=%
\mathcal{A}_{B}$, $\ \forall t\geq 0$;

$\bullet $ $\mathcal{A}_{B}$ attracts the image of $B$ in the strong (weak)
topology, namely, for every neighborhood $\mathcal{O}$\ of $\mathcal{A}_{B}$
in the strong (weak) topology of $E$ there exists a $T=T(\mathcal{O}$ $)>0$
such that $V(t)B\subset \mathcal{O}$ \ for every $t\geq T$. \newline
\end{definition}

\begin{definition}
Let $\{V(t)\}_{{\small t\geq 0}}$ be an operator semigroup on a linear
normed space $E$. A set $\mathcal{A}\subset E$ is called a strong (weak)
global attractor for the semigroup $\left\{ V(t)\right\} _{t\geq 0}$ iff

$\bullet $ $\mathcal{A}$ is strongly (weakly) compact in $E$;

$\bullet $ $\mathcal{A}$ is invariant, i.e. $V(t)\mathcal{A}=\mathcal{A}$, $%
\ \forall t\geq 0$;

$\bullet $ $\mathcal{A}$ attracts the images of all bounded subsets of $E$
in the strong (weak) topology, namely, for every bounded subset $B$ of $E$
and every neighborhood $\mathcal{O}$\ of $\mathcal{A}$ in the strong (weak)
topology of $E$ there exists a $T=T(B,\mathcal{O}$ $)>0$ such that $%
V(t)B\subset \mathcal{O}$ \ for every $t\geq T$. \newline
\end{definition}

By using the method of \cite{15}, one can prove the following well-posedness
result.

\begin{theorem}
Assume that the conditions (2.2) and (2.3) are satisfied. Then, for any $T>0$
and $u_{0}\in W_{0}^{1,p}(0,1),$ $u_{1}\in L^{2}(0,1)$, the problem (2.1)
admits a unique weak solution $u(t,x)$ which satisfies $u\in L^{\infty
}\left( 0,T;W_{0}^{1,p}(0,1)\right) ,$ $u_{t}\in L^{\infty }\left(
0,T;L^{2}(0,1)\right) \cap L^{2}\left( 0,T;H_{0}^{1}(0,1)\right) ,$ $%
u_{tt}\in L^{2}\left( 0,T;W^{-1,\frac{p}{p-1}}(0,1)\right) $ and the energy
inequality%
\begin{equation}
E(u(t))+\int\limits_{s}^{t}\left\Vert u_{tx}\left( \tau \right) \right\Vert
_{L^{2}\left( 0,1\right) }^{2}d\tau \leq E(u(s)),\text{ \ \ }\forall t\geq
s\geq 0.  \tag{2.4}
\end{equation}%
Moreover, if $\ v\in $ $L^{\infty }\left( 0,T;W_{0}^{1,p}(0,1)\right) \cap
W^{1,\infty }\left( 0,T;L^{2}(0,1)\right) \cap W^{1,2}\left(
0,T;H_{0}^{1}(0,1)\right) \cap $\newline
$W^{2,2}(0,T;W^{-1,\frac{p}{p-1}}(0,1))$ is also a weak solution to (2.1)
with initial data $(v_{0},v_{1})\in W_{0}^{1,p}(0,1)\times $\newline
$L^{2}(0,1)$, then
\begin{equation*}
\left\Vert u(t)-v(t)\right\Vert _{H_{0}^{1}(0,1)}+\left\Vert
u_{t}(t)-v_{t}(t)\right\Vert _{H^{-1}(0,1)}
\end{equation*}%
\begin{equation*}
\leq C(T,\left\Vert (u_{0},u_{1})\right\Vert _{W_{0}^{1,p}(0,1)\times
L^{2}(0,1)},\left\Vert (v_{0},v_{1})\right\Vert _{W_{0}^{1,p}(0,1)\times
L^{2}(0,1)})
\end{equation*}%
\begin{equation}
\times \left( \left\Vert u_{0}-v_{0}\right\Vert _{H_{0}^{1}(0,1)}+\left\Vert
u_{1}-v_{1}\right\Vert _{H^{-1}(0,1)}\right) \text{, \ \ }\forall t\in
\lbrack 0,T]\text{,}  \tag{2.5}
\end{equation}%
where $C:R^{+}\times R^{+}\times R^{+}\rightarrow R^{+}$ is a nondecreasing
function with respect to each variable, $E(u(t))=\frac{1}{2}\left\Vert
u_{t}\left( t\right) \right\Vert _{L^{2}\left( 0,1\right) }^{2}+\frac{1}{p}%
\left\Vert u_{x}\left( t\right) \right\Vert _{L^{p}\left( 0,1\right)
}^{p}+\int\limits_{0}^{1}F(u(t,x))dx-\int\limits_{0}^{1}g(x)u(t,x)dx$ and\ $%
F(u)=\int\limits_{0}^{u}f(z)dz$.
\end{theorem}

By Theorem 2.1, it is immediately seen that the problem (2.1) generates a
semigroup $\left\{ S\left( t\right) \right\} _{t\geq 0}$ in $%
W_{0}^{1,p}\left( 0,1\right) \times L^{2}\left( 0,1\right) $, by the rule $%
S\left( t\right) \left( u_{0},u_{1}\right) =\left( u\left( t\right)
,u_{t}\left( t\right) \right) $, where $u(t,x)$ is the weak solution of the
problem (2.1).

Our main results are as follows.

\begin{theorem}
Let the conditions (2.2) and (2.3) hold. Then, for every bounded subset $B$
of $W_{0}^{1,p}\left( 0,1\right) \times L^{2}\left( 0,1\right) $ the
semigroup $\left\{ S\left( t\right) \right\} _{t\geq 0}$, generated by the
problem (2.1), has a weak local attractor $\mathcal{A}_{B}$ in $%
W_{0}^{1,p}(0,1)\times L^{2}(0,1)$. Moreover, the weak local attractor $%
\mathcal{A}_{B}$ attracts the image of $B$ in the strong topology of $%
H_{0}^{1}(0,1)\times L^{2}\left( 0,1\right) $.
\end{theorem}

\begin{theorem}
If, in addition to the conditions (2.2) and (2.3), we assume that $p<4$,
then the semigroup $\left\{ S\left( t\right) \right\} _{t\geq 0}$ possesses
a strong global attractor $\mathcal{A}$ in $W_{0}^{1,p}(0,1)\times
L^{2}(0,1) $ and $\mathcal{A=M}^{u}\left( \mathcal{N}\right) $. Moreover,
the global attractor $\mathcal{A}$ is bounded in $W^{1,\infty }(0,1)\times $
$W^{1,\infty }(0,1)$. Here $\mathcal{M}^{u}\left( \mathcal{N}\right) $ is
unstable manifold emanating from the set of stationary points $\mathcal{N}$
(for definition, see \cite[p. 359]{18}).
\end{theorem}

\begin{remark}
We note that the existence of the strong global attractor in $W^{1,p}\times
L^{2}$, for $p\geq 4$ in one dimensional case and for $p>2$ in
multidimensional case, is still an open question.
\end{remark}

\section{Weak local attractors}

In this section, our aim is to prove the existence of the weak local
attractors in $W_{0}^{1,p}\left( 0,1\right) \times L^{2}\left( 0,1\right) $
for the semigroup $\left\{ S\left( t\right) \right\} _{t\geq 0}$, generated
by the problem (2.1). To this end, we need the following lemmas.

\begin{lemma}
Let the conditions (2.2)-(2.3) hold and $B$ be a bounded subset of $%
W_{0}^{1,p}\left( 0,1\right) \times L^{2}\left( 0,1\right) $. Then every
sequence of the form $\left\{ S\left( t_{k}\right) \varphi _{k}\right\}
_{k=1}^{\infty }$, where $\left\{ \varphi _{k}\right\} _{k=1}^{\infty
}\subset B$, $t_{k}\rightarrow \infty $, has a convergent subsequence in $%
H_{0}^{1}\left( 0,1\right) \times L^{2}\left( 0,1\right) $.
\end{lemma}

\begin{proof}
We first note, by (2.2)-(2.4), that
\begin{equation}
\sup\limits_{t\geq 0}\sup\limits_{\varphi \in B}\left\Vert S\left( t\right)
\varphi \right\Vert _{W_{0}^{1,p}\left( 0,1\right) \times L^{2}\left(
0,1\right) }<\infty .  \tag{3.1}
\end{equation}%
Let $\left( u_{0},u_{1}\right) \in B$ and $\left( u\left( t\right)
,u_{t}\left( t\right) \right) =S\left( t\right) \left( u_{0},u_{1}\right) $.
Denoting $v:=u_{t}$, by (2.1), we obtain%
\begin{equation}
\left\{
\begin{array}{c}
v_{t}+Av=h, \\
v\left( 0\right) =u_{1},%
\end{array}%
\right.  \tag{3.2}
\end{equation}%
where $A:H^{2}(0,1)\cap H_{0}^{1}(0,1)\subset L^{2}\left( 0,1\right)
\rightarrow L^{2}\left( 0,1\right) ,$ $A=-\frac{\partial ^{2}}{\partial x^{2}%
}$ and $h=\frac{\partial }{\partial x}\left( \left\vert u_{x}\right\vert
^{p-2}u_{x}\right) -f\left( u\right) +g$. By the variation of parameters
formula, from (3.2), we have
\begin{equation}
v\left( t\right) =e^{-tA}u_{1}+\int\limits_{0}^{t}e^{-\left( t-\tau \right)
A}h\left( \tau \right) d\tau .  \tag{3.3}
\end{equation}%
By (3.1), it is easy to see that
\begin{equation*}
\left\Vert h\left( t\right) \right\Vert _{H^{-1-\frac{p-2}{2p}}\left(
0,1\right) }\leq c_{1},\text{ \ \ }\forall t\geq 0.
\end{equation*}%
Hence, in (3.3), applying the following well known decay estimate,
\begin{equation*}
\left\Vert e^{-tA}\right\Vert _{\mathcal{L}\left( D(A^{s}),D(A^{\sigma
})\right) }
\end{equation*}%
\begin{equation}
\leq Me^{-\omega t}t^{-(\sigma -s)}\text{, }M\geq 1,\text{ }\omega >0\text{,
}t>0,\text{ }s\leq \sigma ,  \tag{3.4}
\end{equation}%
which can be established for example by the method demonstrated in \cite[p.
116]{19}, and by using $D(A^{\tau })=\left\{
\begin{array}{c}
H^{2\tau }(0,1),\text{ }\tau \in (-\frac{3}{4},\frac{1}{4}] \\
H_{0}^{2\tau }(0,1),\text{ }\tau \in (\frac{1}{4},\frac{3}{4})%
\end{array}%
\right. $,\ we find%
\begin{equation*}
\left\Vert v\left( t\right) \right\Vert _{H^{1-\frac{p-2}{2p}-\delta }\left(
0,1\right) }\leq Mt^{-\frac{1}{2}\left( 1-\frac{p-2}{2p}-\delta \right)
}e^{-\omega t}\left\Vert u_{1}\right\Vert _{L^{2}\left( 0,1\right) }
\end{equation*}%
\begin{equation*}
+M\int\limits_{0}^{t}e^{-\omega \left( t-\tau \right) }\left( t-\tau \right)
^{-1+\frac{\delta }{2}}\left\Vert h\left( \tau \right) \right\Vert _{H^{-1-%
\frac{p-2}{2p}}\left( 0,1\right) }d\tau
\end{equation*}%
\begin{equation*}
\leq Mt^{-\frac{1}{2}\left( 1-\frac{p-2}{2p}-\delta \right) }e^{-\omega
t}\left\Vert u_{1}\right\Vert _{L^{2}\left( 0,1\right)
}+c_{2}\int\limits_{0}^{t}e^{-\omega \left( t-\tau \right) }\left( t-\tau
\right) ^{-1+\frac{\delta }{2}}d\tau \leq \widehat{c}_{\delta },\text{ \ \ \
}\forall t\geq 1,
\end{equation*}%
where $\delta \in \left( 0,1-\frac{p-2}{2p}\right] $. Hence, we have%
\begin{equation}
\sup_{t\geq 1}\left\Vert u_{t}\left( t\right) \right\Vert _{H^{1-\varepsilon
}\left( 0,1\right) }\leq \widetilde{c}_{\varepsilon }\text{, \ for }%
\varepsilon \in \left( \frac{p-2}{2p},1\right] .  \tag{3.5}
\end{equation}

On the other hand, by (2.2)-(2.4) and (3.1), for any $T_{0}\geq 1$, there
exists a subsequence $\left\{ k_{m}\right\} _{m=1}^{\infty }$ such that $%
t_{k_{m}}\geq T_{0}$ and%
\begin{equation}
\left\{
\begin{array}{c}
S\left( t_{k_{m}}-T_{0}\right) \varphi _{k_{m}}\rightarrow \varphi _{0}\text{
weakly in }W_{0}^{1,p}\left( 0,1\right) \times L^{2}\left( 0,1\right) \text{,%
} \\
u_{m}\rightarrow u\text{ weakly star in }L^{\infty }\left( 0,\infty
;W_{0}^{1,p}\left( 0,1\right) \right) \text{, \ \ \ \ \ \ \ \ \ \ } \\
u_{mt}\rightarrow u_{t}\text{ weakly star in }L^{\infty }\left( 0,\infty
;L^{2}\left( 0,1\right) \right) \text{, \ \ \ \ \ \ \ \ \ \ \ \ } \\
u_{mt}\rightarrow u_{t}\text{ weakly in }L^{2}\left( 0,\infty
;H_{0}^{1}\left( 0,1\right) \right) \text{, \ \ \ \ \ \ \ \ \ \ \ \ \ \ \ \
\ \ } \\
u_{m}\left( t\right) \rightarrow u\left( t\right) \text{ weakly in }%
W_{0}^{1,p}\left( 0,1\right) \text{, }\forall t\geq 0\text{, \ \ \ \ \ \ \ \
\ \ \ \ \ } \\
u_{mt}(t)\rightarrow u_{t}(t)\text{ weakly in }L^{2}\left( 0,1\right) \text{%
, }\forall t\geq 0\text{, \ \ \ \ \ \ \ \ \ \ \ \ \ \ \ \ }%
\end{array}%
\right.   \tag{3.6}
\end{equation}%
for some $\varphi _{0}\in W_{0}^{1,p}\left( 0,1\right) \times L^{2}(0,1)$
and $u\in L^{\infty }\left( 0,\infty ;W_{0}^{1,p}\left( 0,1\right) \right)
\cap $ $W^{1,\infty }\left( 0,\infty ;L^{2}(0,1)\right) $, where $\left(
u_{m}(t),u_{mt}(t)\right) =S(t+t_{k_{m}}-T_{0})\varphi _{k_{m}}$. Now,
replacing $u$ in the equation (2.1)$_{1}$ with $u_{m}$ and $u_{n}$, and then
subtracting the obtained equations, we have the following equation:
\begin{equation*}
u_{mtt}(t,x)-u_{ntt}(t,x)-(u_{mtxx}(t,x)-u_{ntxx}(t,x))
\end{equation*}%
\begin{equation*}
-\frac{\partial }{\partial x}(\left\vert u_{mx}(t,x)\right\vert
^{p-2}u_{mx}(t,x)-\left\vert u_{nx}(t,x)\right\vert ^{p-2}u_{nx}(t,x))
\end{equation*}%
\begin{equation}
+f\left( u_{m}(t,x)\right) -f\left( u_{n}(t,x)\right) =0.  \tag{3.7}
\end{equation}%
Testing the equation (3.7) with $2t\left( u_{m}-u_{n}\right) $ in $\left(
0,T\right) \times \left( 0,1\right) $ and considering the inequality%
\begin{equation}
\left( \left\vert x\right\vert ^{p-2}x-\left\vert y\right\vert
^{p-2}y\right) \left( x-y\right) \geq c\left\vert x-y\right\vert ^{p},
\tag{3.8}
\end{equation}%
we find%
\begin{equation*}
T\left\Vert u_{mx}\left( T\right) -u_{nx}\left( T\right) \right\Vert
_{L^{2}\left( 0,1\right) }^{2}+c_{3}\int\limits_{0}^{T}t\left\Vert
u_{mx}\left( t\right) -u_{nx}\left( t\right) \right\Vert _{L^{p}\left(
0,1\right) }^{p}dt
\end{equation*}%
\begin{equation*}
\leq \left\Vert u_{m}\left( T\right) -u_{n}\left( T\right) \right\Vert
_{L^{2}\left( 0,1\right) }^{2}+2\int\limits_{0}^{T}t\left\Vert u_{mt}\left(
t\right) -u_{nt}\left( t\right) \right\Vert _{L^{2}\left( 0,1\right) }^{2}dt
\end{equation*}%
\begin{equation*}
-2T\int\limits_{0}^{1}\left( u_{mt}\left( T,x\right) -u_{nt}\left(
T,x\right) \right) \left( u_{m}\left( T,x\right) -u_{n}\left( T,x\right)
\right) dx+\int\limits_{0}^{T}\left\Vert u_{mx}\left( t\right) -u_{nx}\left(
t\right) \right\Vert _{L^{2}\left( 0,1\right) }^{2}dt
\end{equation*}%
\begin{equation}
-2\int\limits_{0}^{T}\int\limits_{0}^{1}(f\left( u_{m}(t,x)\right) -f\left(
u_{n}(t,x)\right) )t\left( u_{m}\left( t,x\right) -u_{n}\left( t,x\right)
\right) dxdt.  \tag{3.9}
\end{equation}%
Now, considering the fourth term on the right side of (3.9), we get
\begin{equation*}
\int\limits_{0}^{T}\left\Vert u_{mx}\left( t\right) -u_{nx}\left( t\right)
\right\Vert _{L^{2}\left( 0,1\right) }^{2}dt\leq
c_{4}+\int\limits_{1}^{T}\left\Vert u_{mx}\left( t\right) -u_{nx}\left(
t\right) \right\Vert _{L^{2}\left( 0,1\right) }^{2}dt
\end{equation*}%
\begin{equation*}
\leq c_{4}+c_{3}\int\limits_{Q}t\left\Vert u_{mx}\left( t\right)
-u_{nx}\left( t\right) \right\Vert _{L^{2}\left( 0,1\right)
}^{p}dt+\int\limits_{1}^{T}\left( \frac{1}{c_{3}t}\right) ^{\frac{2}{p-2}}dt
\end{equation*}%
\begin{equation}
\leq c_{4}+c_{3}\int\limits_{0}^{T}t\left\Vert u_{mx}\left( t\right)
-u_{nx}\left( t\right) \right\Vert _{L^{p}\left( 0,1\right)
}^{p}dt+c_{5}\left( T^{\max \{0,\frac{p-4}{p-2}\}}+\ln (T)\right) ,\text{ \
\ }\forall T\geq 1,  \tag{3.10}
\end{equation}%
where $Q=\left\{ t\in \left( 0,T\right) :\left\Vert u_{mx}\left( t\right)
-u_{nx}\left( t\right) \right\Vert _{L^{2}\left( 0,1\right) }^{p-2}\geq
\frac{1}{c_{3}t}\right\} .$ By considering (3.10) in (3.9) and using (3.1),
(3.5) and (3.6), we have
\begin{equation*}
\limsup\limits_{m\rightarrow \infty }\limsup\limits_{n\rightarrow \infty
}\left\Vert u_{mx}\left( T\right) -u_{nx}\left( T\right) \right\Vert
_{L^{2}\left( 0,1\right) }^{2}\leq \frac{c_{6}\left( T^{\max \{0,\frac{p-4}{%
p-2}\}}+\ln (T)\right) }{T},\text{ }\ \text{ }\forall T\geq 1.
\end{equation*}%
Choosing $T=T_{0}$ in the above inequality, we obtain%
\begin{equation*}
\liminf\limits_{k\rightarrow \infty }\liminf\limits_{n\rightarrow \infty
}\left\Vert PS\left( t_{k}\right) \varphi _{k}-PS\left( t_{n}\right) \varphi
_{n}\right\Vert _{H_{0}^{1}\left( 0,1\right) }\text{ }\leq \frac{c_{7}\left(
T_{0}^{\max \{0,\frac{p-4}{p-2}\}}+\ln (T_{0})\right) ^{\frac{1}{2}}}{T_{0}^{%
\frac{1}{2}}},\text{ \ \ }\forall T_{0}\geq 1\text{, }
\end{equation*}%
and passing to limit as $T_{0}\rightarrow \infty $, we get%
\begin{equation*}
\liminf\limits_{k\rightarrow \infty }\liminf\limits_{n\rightarrow \infty
}\left\Vert PS\left( t_{k}\right) \varphi _{k}-PS\left( t_{n}\right) \varphi
_{n}\right\Vert _{H_{0}^{1}\left( 0,1\right) }=0\text{,}
\end{equation*}%
where $P:W_{0}^{1,p}\left( 0,1\right) \times L^{2}\left( 0,1\right)
\rightarrow W_{0}^{1,p}\left( 0,1\right) $ is the projector defined by $%
P\left( \varphi ,\psi \right) =\varphi $. Also, it can be immediately seen
that for every subsequence $\left\{ k_{m}\right\} _{m=1}^{\infty }$, the
following holds:%
\begin{equation}
\liminf\limits_{m\rightarrow \infty }\liminf\limits_{n\rightarrow \infty
}\left\Vert PS\left( t_{k_{m}}\right) \varphi _{k_{m}}-PS\left(
t_{k_{n}}\right) \varphi _{k_{n}}\right\Vert _{H_{0}^{1}\left( 0,1\right)
}=0.  \tag{3.11}
\end{equation}%
Now, we conclude that the sequence $\left\{ PS\left( t_{k}\right) \varphi
_{k}\right\} _{k=1}^{\infty }$ has a convergent subsequence in $%
H_{0}^{1}\left( 0,1\right) $. If we assume the contrary, then, by the
completeness of $H_{0}^{1}\left( 0,1\right) $, there exist $\varepsilon
_{0}>0$ and a subsequence $\left\{ k_{m}\right\} _{m=1}^{\infty }$ such that%
\begin{equation*}
\left\Vert PS\left( t_{k_{m}}\right) \varphi _{k_{m}}-PS\left(
t_{k_{n}}\right) \varphi _{k_{n}}\right\Vert _{H_{0}^{1}\left( 0,1\right)
}\geq \varepsilon _{0},\text{ }m\neq n,
\end{equation*}%
which contradicts (3.11). Hence, together with (3.5), we complete the proof
of the lemma.
\end{proof}

Now, we define weak $\omega $-limit set of the trajectories emanating from a
set $B\subset W_{0}^{1,p}\left( 0,1\right) \times L^{2}\left( 0,1\right) $
as follows:%
\begin{equation*}
\omega _{w}\left( B\right) :=\bigcap\limits_{\tau \geq 0}\overline{%
\bigcup\limits_{t\geq \tau }S\left( t\right) B}^{w},
\end{equation*}%
where the bar over a set means weak closure in $W_{0}^{1,p}\left( 0,1\right)
\times L^{2}\left( 0,1\right) $. It can be easily shown that $\varphi \in
\omega _{w}\left( B\right) $ if and only if there exist sequences $\left\{
t_{k}\right\} _{k=1}^{\infty }$, $t_{k}\rightarrow \infty $ and $\left\{
\varphi _{k}\right\} _{k=1}^{\infty }\subset B$ such that $S\left(
t_{k}\right) \varphi _{k}\rightarrow \varphi $ weakly in $W_{0}^{1,p}\left(
0,1\right) \times L^{2}\left( 0,1\right) $. Moreover, we state the following
invariance property of the set $\omega _{w}\left( B\right) $:

\begin{lemma}
For any bounded set $B\subset W_{0}^{1,p}\left( 0,1\right) \times
L^{2}\left( 0,1\right) $, the set $\omega _{w}\left( B\right) $ is invariant.
\end{lemma}

\begin{proof}
Let $\psi \in \omega _{w}\left( B\right) $ and $z=S\left( t\right) \psi $
for $t\geq 0$. Then, by the definition of $\omega _{w}\left( B\right) $,
there exist the sequences $\left\{ t_{k}\right\} _{k=1}^{\infty }$, $%
t_{k}\rightarrow \infty $ and $\left\{ \psi _{k}\right\} _{k=1}^{\infty
}\subset B$ such that $S\left( t_{k}\right) \psi _{k}\rightarrow \psi $
weakly in $W_{0}^{1,p}\left( 0,1\right) \times L^{2}\left( 0,1\right) $.
Also, by Lemma 3.1, there exists a subsequence $\left\{ k_{m}\right\}
_{m=1}^{\infty }$ such that $PS\left( t_{k_{m}}\right) \psi
_{k_{m}}\rightarrow P\psi $ strongly in $H_{0}^{1}\left( 0,1\right) $.
Therefore, setting $\tau _{k_{m}}:=t+t_{k_{m}}$, by (2.5), we have
\begin{equation*}
S\left( \tau _{k_{m}}\right) \psi _{k_{m}}=S\left( t\right) S\left(
t_{k_{m}}\right) \psi _{k_{m}}\rightarrow S\left( t\right) \psi =z\text{ \
weakly in }W_{0}^{1,p}\left( 0,1\right) \times L^{2}\left( 0,1\right) ,
\end{equation*}%
which yields that $z\in \omega _{w}\left( B\right) $. Hence, we have that $%
S\left( t\right) \omega _{w}\left( B\right) \subset \omega _{w}\left(
B\right) $.

On the other hand, if $\psi \in \omega _{w}\left( B\right) $, then there
exist $\left\{ t_{k}\right\} _{k=1}^{\infty }$, $t_{k}\rightarrow \infty $
and $\left\{ \psi _{k}\right\} _{k=1}^{\infty }\subset B$ such that $S\left(
t_{k}\right) \psi _{k}\rightarrow \psi $ weakly in $W_{0}^{1,p}\left(
0,1\right) \times L^{2}\left( 0,1\right) $. Now, define $\varphi
_{k}=S\left( t_{k}-t\right) \psi _{k}$, for $t_{k}\geq t\geq 0$. By (3.1),
there exists a subsequence $\left\{ k_{m}\right\} _{m=1}^{\infty }$ such
that $\varphi _{k_{m}}\rightarrow \varphi $ weakly in $W_{0}^{1,p}\left(
0,1\right) \times L^{2}\left( 0,1\right) $ for some $\varphi \in
W_{0}^{1,p}\left( 0,1\right) \times L^{2}\left( 0,1\right) $, which gives
that $\varphi \in \omega _{w}\left( B\right) $. Moreover, by Lemma 3.1,
passing to a subsequence, we have $P\varphi _{k_{m_{n}}}\rightarrow P\varphi
$ strongly in $H_{0}^{1}\left( 0,1\right) .$ Since
\begin{equation*}
S\left( t_{k_{m_{n}}}\right) \psi _{k_{m_{n}}}=S\left( t\right) S\left(
t_{k_{m_{n}}}-t\right) \psi _{k_{m_{n}}}=S\left( t\right) \varphi
_{k_{m_{n}}},
\end{equation*}%
applying (2.5), we observe that $S\left( t_{k_{m_{n}}}\right) \psi
_{k_{m_{n}}}\rightarrow S\left( t\right) \varphi $ weakly in $%
W_{0}^{1,p}\left( 0,1\right) \times L^{2}\left( 0,1\right) .$ Therefore, we
conclude that $\psi =S\left( t\right) \varphi $, which gives that $\omega
_{w}\left( B\right) \subset $ $S\left( t\right) \omega _{w}\left( B\right) $.
\end{proof}

Thus, Lemma 3.1 and Lemma 3.2 imply Theorem 2.2.

\section{Regular strong global attractor}

We begin with the following regularity result.

\begin{lemma}
In addition to the conditions (2.2)-(2.3), assume that $p<4$ and $B$ is a
bounded subset of $W_{0}^{1,p}\left( 0,1\right) \times L^{2}\left(
0,1\right) $. Then, the set $\omega _{w}\left( B\right) $ is bounded in $%
W^{1,\infty }\left( 0,1\right) \times W^{1,\infty }\left( 0,1\right) $.
\end{lemma}

\begin{proof}
Let $\left( u_{0},u_{1}\right) \in \omega _{w}\left( B\right) $. By the
invariance of $\omega _{w}\left( B\right) $, it follows that there exists an
invariant trajectory (see \cite[p. 157]{20}) $\left\{ \left( u\left(
t\right) ,u_{t}\left( t\right) \right) :t\in
\mathbb{R}
\right\} \subset \omega _{w}\left( B\right) $ such that
\begin{equation}
\left( u\left( 0\right) ,u_{t}\left( 0\right) \right) =\left(
u_{0},u_{1}\right) .  \tag{4.1}
\end{equation}%
Since $u$ is the solution of (2.1)$_{1}$, by (2.2)-(2.4), we have%
\begin{equation}
\left\Vert \left( u\left( t\right) ,u_{t}\left( t\right) \right) \right\Vert
_{W_{0}^{1,p}\left( 0,1\right) \times L^{2}\left( 0,1\right)
}+\int\limits_{s}^{t}\left\Vert u_{tx}\left( \tau \right) \right\Vert
_{L^{2}\left( 0,1\right) }^{2}d\tau \leq c_{1}\text{, \ \ }\forall t\geq s,
\tag{4.2}
\end{equation}%
where the constant $c_{1}$ depends on $\omega _{w}\left( B\right) $ and is
independent of the trajectory $\left\{ \left( u\left( t\right) ,u_{t}\left(
t\right) \right) :t\in
\mathbb{R}
\right\} $.

Now, denoting $v\left( t\right) :=u_{t}\left( t\right) $ as in the proof of
Lemma 3.1, by (2.1)$_{1}$, we have%
\begin{equation}
v_{t}+Av=h.  \tag{4.3}
\end{equation}%
Hence, similar to (3.5), we have%
\begin{equation*}
\left\Vert v\left( t\right) \right\Vert _{H^{1-\frac{p-2}{2p}-\delta }\left(
0,1\right) }\leq c_{2}\left( \left( t-s\right) ^{-\frac{1}{2}\left( 1-\frac{%
p-2}{2p}-\delta \right) }+1\right) \text{, \ \ }\forall t\geq s,
\end{equation*}%
where $\delta \in \left( 0,1-\frac{p-2}{2p}\right] $. Passing to the limit
as $s\rightarrow -\infty $ in the above inequality, we obtain%
\begin{equation}
\left\Vert u_{tx}\left( t\right) \right\Vert _{H^{-\varepsilon }\left(
0,1\right) }\leq c_{3}\text{, \ }\forall t\in
\mathbb{R}
,  \tag{4.4}
\end{equation}%
where $\varepsilon \in \left( \frac{p-2}{2p},1\right] $ and the constant $%
c_{3}$, as the previous constants, is independent of the trajectory $\left\{
\left( u\left( t\right) ,u_{t}\left( t\right) \right) :t\in
\mathbb{R}
\right\} $.

Similarly, denoting $\widehat{v}\left( t\right) :=v_{t}\left( t\right) $,
from (4.3), we have%
\begin{equation}
\widehat{v}\left( t\right) =e^{-A\left( t-s\right) }\widehat{v}\left(
s\right) +\int\limits_{s}^{t}e^{-A\left( t-\tau \right) }h^{\prime }\left(
\tau \right) d\tau ,\text{ \ \ }\forall t\geq s.  \tag{4.5}
\end{equation}%
By (4.2), it is easy to verify that%
\begin{equation}
\left\Vert h^{\prime }\left( t\right) \right\Vert _{H^{-1-\frac{p-2}{p}%
}\left( 0,1\right) }\leq c_{4}\left\Vert u_{tx}\left( t\right) \right\Vert
_{L^{2}\left( 0,1\right) },\text{ \ \ }\forall t\in
\mathbb{R}
.  \tag{4.6}
\end{equation}%
So, applying (3.4) to (4.5) and considering (4.2) and (4.6), we get%
\begin{equation*}
\left\Vert \widehat{v}\left( t\right) \right\Vert _{H^{-\frac{p-2}{p}-\delta
}\left( 0,1\right) }\leq c_{5}\left( t-s\right) ^{-1+\frac{1}{2}\left( \frac{%
p-2}{p}+\delta \right) }\left\Vert \widehat{v}\left( s\right) \right\Vert
_{D(A^{-1})}
\end{equation*}%
\begin{equation*}
+c_{5}\int\limits_{s}^{t}e^{-\omega \left( t-\tau \right) }\left( t-\tau
\right) ^{-\frac{1-\delta }{2}}\left\Vert u_{tx}\left( \tau \right)
\right\Vert _{L^{2}\left( 0,1\right) }d\tau
\end{equation*}%
\begin{equation*}
\leq c_{5}\left( t-s\right) ^{-1+\frac{1}{2}\left( \frac{p-2}{p}+\delta
\right) }\left\Vert u_{tt}\left( s\right) \right\Vert _{D(A^{-1})}
\end{equation*}%
\begin{equation*}
+c_{5}\left( \int\limits_{s}^{t}\left\Vert u_{tx}\left( \tau \right)
\right\Vert _{L^{2}\left( 0,1\right) }^{2}d\tau \right) ^{\frac{1}{2}}\left(
\int\limits_{s}^{t}e^{-2\omega \left( t-\tau \right) }\left( t-\tau \right)
^{-\left( 1-\delta \right) }d\tau \right) ^{\frac{1}{2}}
\end{equation*}%
\begin{equation*}
\leq c_{6}\left( \left( t-s\right) ^{-1+\frac{1}{2}\left( \frac{p-2}{p}%
+\delta \right) }+1\right) ,\text{\ \ \ }\forall t\geq s,
\end{equation*}%
where $\delta \in (0,1]$ and the constant $c_{6}$, as the previous constants
$c_{i}$ ($i=\overline{1,5}$), is independent of the trajectory $\left\{
\left( u\left( t\right) ,u_{t}\left( t\right) \right) :t\in
\mathbb{R}
\right\} $. Passing to the limit as $s\rightarrow -\infty $ in the last
inequality, we obtain%
\begin{equation}
\left\Vert u_{tt}\left( t\right) \right\Vert _{H^{-\varepsilon }\left(
0,1\right) }\leq c_{6},\text{ \ \ }\forall t\in
\mathbb{R}
,  \tag{4.7}
\end{equation}%
where $\varepsilon \in (\frac{p-2}{p},\frac{p-2}{p}+1]$.

Now, denoting $w\left( t,x\right) :=u_{x}\left( t-1,x\right) $, by (2.1)$%
_{1} $, we find%
\begin{equation}
\frac{\partial }{\partial x}\left( w_{t}\left( t,x\right) +\left\vert
w\left( t,x\right) \right\vert ^{p-2}w\left( t,x\right) \right) =\kappa
\left( t,x\right) ,  \tag{4.8}
\end{equation}%
where $\kappa \left( t,x\right) :=u_{tt}\left( t-1,x\right) +f\left( u\left(
t-1,x\right) \right) -g\left( x\right) $. Choosing $\varepsilon \in \left(
\frac{p-2}{p},\frac{1}{2}\right) $, by (4.2), (4.4) and (4.7), we have%
\begin{equation*}
\left( w_{t}+\left\vert w\right\vert ^{p-2}w\right) \in L^{\infty }\left(
\mathbb{R}
;H^{-\varepsilon }\left( 0,1\right) \right) \text{ and }\kappa \in L^{\infty
}\left(
\mathbb{R}
;H^{-\varepsilon }\left( 0,1\right) \right) .
\end{equation*}%
Hence, applying Lemma A.1, by (4.8), we obtain that $\left( w_{t}+\left\vert
w\right\vert ^{p-2}w\right) \in L^{\infty }\left(
\mathbb{R}
;C\left[ 0,1\right] \right) $ and%
\begin{equation}
\left\vert w_{t}\left( t,x\right) +\left\vert w\left( t,x\right) \right\vert
^{p-2}w\left( t,x\right) \right\vert \leq \widehat{\kappa }\left( t\right)
\text{, \ \ }\forall \left( t,x\right) \in
\mathbb{R}
\times \left[ 0,1\right] ,  \tag{4.9}
\end{equation}%
where $\widehat{\kappa }\left( t\right) :=\left\Vert w_{t}\left( t\right)
+\left\vert w\left( t\right) \right\vert ^{p-2}w\left( t\right) \right\Vert
_{H^{-\varepsilon }\left( 0,1\right) }+\left\Vert \kappa \left( t\right)
\right\Vert _{H^{-\varepsilon }\left( 0,1\right) }$. As mentioned above, by
(4.2), (4.4) and (4.7), it follows that%
\begin{equation*}
\widehat{\kappa }\left( t\right) \leq c_{7}\text{, \ }\forall t\in
\mathbb{R}
.
\end{equation*}%
Therefore, applying Lemma A.2 to (4.9), we get%
\begin{equation*}
\left\vert w_{t}\left( t,\cdot \right) \right\vert +\left\vert w\left(
t,\cdot \right) \right\vert \leq c_{8}\text{,\ \ \ a.e. \ in \ }\left[ 0,1%
\right] ,
\end{equation*}%
for every $t\in \left[ s_{0},1\right] $, where $s_{0}\in (\frac{2}{p},1)$.
Choosing $t=1$ in the last inequality, by (4.1) and the definition of $%
w\left( t,x\right) $, we obtain
\begin{equation*}
\left\Vert u_{0x}\right\Vert _{L^{\infty }\left( 0,1\right) }+\left\Vert
u_{1x}\right\Vert _{L^{\infty }\left( 0,1\right) }\leq c_{8}.
\end{equation*}%
The last inequality completes the proof, because the constant $c_{8}$ is
independent of $\left( u_{0},u_{1}\right) $.
\end{proof}

Now, we are in a position to prove the following asymptotic compactness
result.

\begin{lemma}
Let the conditions of Lemma 4.1 hold. Then every sequence of the form $%
\left\{ S\left( t_{k}\right) \varphi _{k}\right\} _{k=1}^{\infty }$, where $%
\left\{ \varphi _{k}\right\} _{k=1}^{\infty }\subset B$, $t_{k}\rightarrow
\infty $, has a convergent subsequence in $W_{0}^{1,p}\left( 0,1\right)
\times L^{2}\left( 0,1\right) $.
\end{lemma}

\begin{proof}
Since Lemma 3.1 implies that $\left\{ S\left( t_{k}\right) \varphi
_{k}\right\} _{k=1}^{\infty }$ has a convergent subsequence in $%
H_{0}^{1}\left( 0,1\right) \times L^{2}\left( 0,1\right) $, it is sufficient
to prove that $\left\{ PS\left( t_{k}\right) \varphi _{k}\right\}
_{k=1}^{\infty }$ admits a convergent subsequence in $W_{0}^{1,p}\left(
0,1\right) $. Let $u_{m}(t,x)$ and $u(t,x)$ are the same functions as in the
proof of Lemma 3.1. Then, by Lemma 3.1, beside (3.6), we also have%
\begin{equation*}
S\left( t_{k_{m}}-T_{0}\right) \varphi _{k_{m}}\rightarrow \varphi _{0}\in
\omega _{w}(B)\text{ \ strongly in }H_{0}^{1}\left( 0,1\right) \times
L^{2}\left( 0,1\right) \text{,}
\end{equation*}%
and consequently, by (2.5),%
\begin{equation}
(u(t),u_{t}(t))=S(t)\varphi _{0}\text{ and\ }u_{m}\left( t\right)
\rightarrow u\left( t\right) \text{ strongly in }H_{0}^{1}\left( 0,1\right) ,%
\text{\ \ }\forall t\geq 0\text{.}  \tag{4.10}
\end{equation}%
Now, putting $u_{m}$ instead of $u$ in (2.1), by (2.2)-(2.4), we have%
\begin{equation*}
\left\Vert u_{m}\left( t\right) \right\Vert _{W^{1,p}\left( 0,1\right)
}+\left\Vert u_{mt}\left( t\right) \right\Vert _{L^{2}\left( 0,1\right)
}+\int\limits_{0}^{t}\left\Vert u_{mtx}\left( \tau \right) \right\Vert
_{L^{2}\left( 0,1\right) }^{2}d\tau \leq c_{1},\text{ \ }\forall t\geq 0.
\end{equation*}%
Also, putting $u_{m}$ instead of $u$ in (2.1) and testing the obtained
equation by $u_{m}$ in $\left( 0,t\right) \times \left( 0,1\right) $, we find%
\begin{equation*}
\int\limits_{0}^{t}\left\Vert u_{mx}\left( \tau \right) \right\Vert
_{L^{p}\left( 0,1\right) }^{p}d\tau
+\int\limits_{0}^{t}\int\limits_{0}^{1}f\left( u_{m}\left( \tau ,x\right)
\right) u_{m}\left( \tau ,x\right) dxd\tau
\end{equation*}%
\begin{equation*}
-\int\limits_{0}^{t}\int\limits_{0}^{1}g\left( x\right) u_{m}\left( \tau
,x\right) dxd\tau =\int\limits_{0}^{t}\left\Vert u_{mt}\left( \tau \right)
\right\Vert _{L^{2}\left( 0,1\right) }^{2}d\tau
\end{equation*}%
\begin{equation*}
-\int\limits_{0}^{1}u_{mt}\left( t,x\right) u_{m}\left( t,x\right)
dx+\int\limits_{0}^{1}u_{mt}\left( 0,x\right) u_{m}\left( 0,x\right) dx
\end{equation*}%
\begin{equation*}
-\frac{1}{2}\left\Vert u_{mx}\left( t\right) \right\Vert _{L^{2}\left(
0,1\right) }^{2}+\frac{1}{2}\left\Vert u_{mx}\left( 0\right) \right\Vert
_{L^{2}\left( 0,1\right) }^{2},\text{ \ \ }\forall t\geq 0,
\end{equation*}%
which, together with the last inequality, yields%
\begin{equation*}
\left\vert \frac{1}{2}\int\limits_{0}^{t}\left\Vert u_{mt}\left( \tau
\right) \right\Vert _{L^{2}\left( 0,1\right) }^{2}d\tau +\frac{1}{p}%
\int\limits_{0}^{t}\left\Vert u_{mx}\left( \tau \right) \right\Vert
_{L^{p}\left( 0,1\right) }^{p}d\tau \right.
\end{equation*}%
\begin{equation}
+\left. \frac{1}{p}\int\limits_{0}^{t}\int\limits_{0}^{1}f\left( u_{m}\left(
\tau ,x\right) \right) u_{m}\left( \tau ,x\right) dxd\tau -\frac{1}{p}%
\int\limits_{0}^{t}\int\limits_{0}^{1}g\left( x\right) u_{m}\left( \tau
,x\right) dxd\tau \right\vert \leq c_{2}\text{, \ \ }\forall t\geq 0.
\tag{4.11}
\end{equation}%
Similarly, for $u\left( t,x\right) $, we have%
\begin{equation*}
\left\vert \frac{1}{2}\int\limits_{0}^{t}\left\Vert u_{t}\left( \tau \right)
\right\Vert _{L^{2}\left( 0,1\right) }^{2}d\tau +\frac{1}{p}%
\int\limits_{0}^{t}\left\Vert u_{x}\left( \tau \right) \right\Vert
_{L^{p}\left( 0,1\right) }^{p}d\tau \right.
\end{equation*}%
\begin{equation}
\left. +\frac{1}{p}\int\limits_{0}^{t}\int\limits_{0}^{1}f\left( u\left(
\tau ,x\right) \right) u\left( \tau ,x\right) dxd\tau -\frac{1}{p}%
\int\limits_{0}^{t}\int\limits_{0}^{1}g\left( x\right) u\left( \tau
,x\right) dxd\tau \right\vert \leq c_{3}\text{, \ \ }\forall t\geq 0.
\tag{4.12}
\end{equation}%
By (4.11)-(4.12), we get%
\begin{equation*}
\frac{1}{2}\int\limits_{0}^{t}\left\Vert u_{mt}\left( \tau \right)
\right\Vert _{L^{2}\left( 0,1\right) }^{2}d\tau +\frac{1}{p}%
\int\limits_{0}^{t}\left\Vert u_{mx}\left( \tau \right) \right\Vert
_{L^{p}\left( 0,1\right) }^{p}d\tau
\end{equation*}%
\begin{equation*}
\leq c_{2}+c_{3}+\frac{1}{2}\int\limits_{0}^{t}\left\Vert u_{t}\left( \tau
\right) \right\Vert _{L^{2}\left( 0,1\right) }^{2}d\tau +\frac{1}{p}%
\int\limits_{0}^{t}\left\Vert u_{x}\left( \tau \right) \right\Vert
_{L^{p}\left( 0,1\right) }^{p}d\tau
\end{equation*}%
\begin{equation*}
+\frac{1}{p}\int\limits_{0}^{t}\int\limits_{0}^{1}\left[ f\left( u\left(
\tau ,x\right) \right) u\left( \tau ,x\right) -g\left( x\right) u\left( \tau
,x\right) \right.
\end{equation*}%
\begin{equation*}
\left. -f\left( u_{m}\left( \tau ,x\right) \right) u_{m}\left( \tau
,x\right) +g\left( x\right) u_{m}\left( \tau ,x\right) \right] dxd\tau ,
\end{equation*}%
and consequently%
\begin{equation}
\int\limits_{0}^{t}E\left( u_{m}\left( \tau \right) \right) d\tau \leq
c_{2}+c_{3}+\int\limits_{0}^{t}E\left( u\left( \tau \right) \right) d\tau
+\Lambda _{m}\left( t\right) \text{, \ \ }\forall t\geq 0,  \tag{4.13}
\end{equation}%
where
\begin{equation*}
\Lambda _{m}\left( t\right) :=\frac{1}{p}\int\limits_{0}^{t}\int%
\limits_{0}^{1}\left[ f\left( u\left( \tau ,x\right) \right) u\left( \tau
,x\right) -f\left( u_{m}\left( \tau ,x\right) \right) u_{m}\left( \tau
,x\right) \right] dxd\tau
\end{equation*}%
\begin{equation*}
+\frac{p-1}{p}\int\limits_{0}^{t}\int\limits_{0}^{1}g\left( x\right) \left(
u\left( \tau ,x\right) -u_{m}\left( \tau ,x\right) \right) dxd\tau
\end{equation*}%
\begin{equation*}
+\int\limits_{0}^{t}\int\limits_{0}^{1}\left[ F\left( u_{m}\left( \tau
,x\right) \right) -F\left( u\left( \tau ,x\right) \right) \right] dxd\tau .
\end{equation*}%
By (3.6), it is easy to see that%
\begin{equation*}
\lim\limits_{m\rightarrow \infty }\Lambda _{m}\left( t\right) =0\text{, \ \ }%
\forall t\geq 0\text{.}
\end{equation*}%
Hence, passing to the limit in (4.13), we obtain%
\begin{equation}
\limsup\limits_{m\rightarrow \infty }\int\limits_{0}^{t}E\left( u_{m}\left(
\tau \right) \right) d\tau \leq c_{4}+\int\limits_{0}^{t}E\left( u\left(
\tau \right) \right) d\tau \text{, \ \ }\forall t\geq 0.  \tag{4.14}
\end{equation}%
Since $\varphi _{0}\in \omega _{w}\left( B\right) $, by the invariance of $\
\omega _{w}\left( B\right) $ and Lemma 4.1, we have $\left( u\left( t\right)
,u_{t}\left( t\right) \right) \in $ $\omega _{w}\left( B\right) \subset
W^{1,\infty }\left( 0,1\right) \times W^{1,\infty }\left( 0,1\right) $.
Hence, testing (2.1) by $u_{t}$ in $\left( 0,t\right) \times \left(
0,1\right) $, we find the energy equality%
\begin{equation}
E\left( u\left( t\right) \right) +\int\limits_{\tau }^{t}\left\Vert
u_{xt}\left( s\right) \right\Vert _{L^{2}\left( 0,1\right) }^{2}ds=E\left(
u\left( \tau \right) \right) \text{, \ }0\leq \tau \leq t.  \tag{4.15}
\end{equation}%
Now, applying the energy inequality (2.4) to the left hand side of (4.14)
and the energy equality (4.15) to the right hand side of (4.14), and taking
into account (3.6) , we get%
\begin{equation*}
\limsup\limits_{m\rightarrow \infty }tE\left( u_{m}\left( t\right) \right)
\leq c_{4}+tE\left( u\left( t\right) \right) ,
\end{equation*}%
and consequently%
\begin{equation}
\limsup\limits_{m\rightarrow \infty }\left\Vert u_{mx}\left( t\right)
\right\Vert _{L^{p}\left( 0,1\right) }^{p}\leq \frac{pc_{4}}{t}+\left\Vert
u_{x}\left( t\right) \right\Vert _{L^{p}\left( 0,1\right) }^{p},\text{ \ }%
\forall t>0.  \tag{4.16}
\end{equation}%
By using (3.6)$_{5}$ and (4.10), we also obtain%
\begin{equation*}
\limsup\limits_{m\rightarrow \infty }\int\limits_{0}^{1}\left( \left\vert
u_{mx}\left( t,x\right) \right\vert ^{p-2}u_{mx}\left( t,x\right)
-\left\vert u_{x}\left( t,x\right) \right\vert ^{p-2}u_{x}\left( t,x\right)
\right)
\end{equation*}%
\begin{equation}
\times \left( u_{mx}\left( t,x\right) -u_{x}\left( t,x\right) \right)
dx=\limsup\limits_{m\rightarrow \infty }\left\Vert u_{mx}\left( t\right)
\right\Vert _{L^{p}\left( 0,1\right) }^{p}-\left\Vert u_{x}\left( t\right)
\right\Vert _{L^{p}\left( 0,1\right) }^{p}.  \tag{4.17}
\end{equation}%
Therefore, taking into account (3.8), by (4.16)-(4.17), we have%
\begin{equation*}
\limsup\limits_{m\rightarrow \infty }\left\Vert u_{mx}\left( t\right)
-u_{x}\left( t\right) \right\Vert _{L^{p}\left( 0,1\right) }\leq \frac{c_{5}%
}{t^{\frac{1}{p}}},\text{\ \ }\forall t>0,
\end{equation*}%
and consequently%
\begin{equation*}
\limsup\limits_{k\rightarrow \infty }\limsup\limits_{m\rightarrow \infty
}\left\Vert u_{mx}\left( t\right) -u_{kx}\left( t\right) \right\Vert
_{L^{p}\left( 0,1\right) }\leq \frac{2c_{5}}{t^{\frac{1}{p}}},\text{ \ }%
\forall t>0.
\end{equation*}%
Taking $t=T_{0}$ in the last inequality, we get$^{{}}$%
\begin{equation*}
\liminf\limits_{k\rightarrow \infty }\liminf\limits_{m\rightarrow \infty
}\left\Vert PS\left( t_{m}\right) \varphi _{m}-PS\left( t_{k}\right) \varphi
_{k}\right\Vert _{W_{0}^{1,p}\left( 0,1\right) }=0.
\end{equation*}%
Thus, repeating the arguments done at the end of the proof of Lemma 3.1, we
obtain that the sequence $\left\{ P\left( S\left( t_{k}\right) \right)
\varphi _{k}\right\} _{k=1}^{\infty }$ has a convergent subsequence in $%
W_{0}^{1,p}\left( 0,1\right) $.
\end{proof}

\bigskip From (2.4), it follows that the problem (2.1) admits a strict
Lyapunov function%
\begin{equation*}
L\left( u,v\right) =\frac{1}{2}\left\Vert v\right\Vert _{L^{2}\left(
0,1\right) }^{2}+\frac{1}{p}\left\Vert u_{x}\right\Vert _{L^{p}\left(
0,1\right) }^{p}+\int\limits_{0}^{1}F(u(x))dx-\int\limits_{0}^{1}g\left(
x\right) u\left( x\right) dx
\end{equation*}%
in $W_{0}^{1,p}\left( 0,1\right) \times L^{2}\left( 0,1\right) $. Hence, by
Lemma 4.1, Lemma 4.2 and [18, Corollary 7.5.7], we obtain Theorem 2.3.
\bigskip

\appendix

\section{\protect\bigskip}

\begin{lemma}
If $f\in H^{-\varepsilon }\left( 0,1\right) $ and $f^{\prime }\in
H^{-\varepsilon }\left( 0,1\right) $, then $f\in C\left[ 0,1\right] $ and
\begin{equation}
\left\Vert f\right\Vert _{C\left[ 0,1\right] }\leq c\left( \left\Vert
f\right\Vert _{H^{-\varepsilon }\left( 0,1\right) }+\left\Vert f^{\prime
}\right\Vert _{H^{-\varepsilon }\left( 0,1\right) }\right) ,  \tag{A.1}
\end{equation}%
where $\varepsilon \in \lbrack 0,\frac{1}{2})$.
\end{lemma}

\begin{proof}
Firstly, let us prove density of $\mathcal{D}\left[ 0,1\right] $ in the
linear normed space $X=\left\{ f:f\in H^{-\varepsilon }\left( 0,1\right)
,\right. $\newline
$\left. f^{\prime }\in H^{-\varepsilon }\left( 0,1\right) \right\} $ endowed
with the norm
\begin{equation*}
\left\Vert f\right\Vert _{X}=\left\Vert f\right\Vert _{H^{-\varepsilon
}\left( 0,1\right) }+\left\Vert f^{\prime }\right\Vert _{H^{-\varepsilon
}\left( 0,1\right) }.
\end{equation*}%
Let us define a linear continuous functional $\phi $ on $X$ such that%
\begin{equation*}
\phi \left( v\right) :=\left\langle u_{0},v\right\rangle +\left\langle
u_{1},v^{\prime }\right\rangle ,
\end{equation*}%
where $u_{0},u_{1}\in H^{\varepsilon }\left( 0,1\right) $. Assume that%
\begin{equation}
\phi \left( \varphi \right) =0,  \tag{A.2}
\end{equation}%
for every $\varphi \in $ $\mathcal{D}\left[ 0,1\right] $. To prove that $%
\overline{\mathcal{D}\left[ 0,1\right] }^{X}=X$, it is sufficient to show
that%
\begin{equation*}
\phi \left( v\right) =0,
\end{equation*}%
for every $v\in X.$ Let%
\begin{equation*}
\widehat{u}_{i}\left( x\right) :=\left\{
\begin{array}{l}
u_{i}\left( x\right) ,\text{ }x\in \left( 0,1\right) , \\
0,\text{ }%
\mathbb{R}
\backslash \left( 0,1\right)%
\end{array}%
\right. ,\text{ }i=0,1,
\end{equation*}%
and%
\begin{equation*}
\widehat{\varphi }\left( x\right) :=\left\{
\begin{array}{l}
\varphi \left( x\right) ,\text{ }x\in \left[ 0,1\right] , \\
\varphi \left( 1\right) ,\text{ }x>1, \\
\varphi \left( 0\right) ,\text{ }x<0,%
\end{array}%
\right.
\end{equation*}%
where $\varphi \in $ $\mathcal{D}\left[ 0,1\right] $. In addition, let us
denote $\widetilde{\varphi }\left( x\right) :=\rho \left( x\right) \widehat{%
\varphi }\left( x\right) $, where $\rho \in \mathcal{D}\left(
\mathbb{R}
\right) $ and $\rho \left( x\right) =1$, for $x\in \left[ 0,1\right] $.
Since $\widetilde{\varphi }\in H^{1}\left(
\mathbb{R}
\right) $, by (A.2), it follows that%
\begin{equation*}
\int\limits_{%
\mathbb{R}
}\widehat{u}_{0}\left( x\right) \widetilde{\varphi }\left( x\right)
dx+\int\limits_{%
\mathbb{R}
}\widehat{u}_{1}\left( x\right) \frac{d}{dx}\widetilde{\varphi }\left(
x\right) dx=0,
\end{equation*}%
and consequently%
\begin{equation*}
\frac{d}{dx}\widehat{u}_{1}\left( x\right) =\widehat{u}_{0}\left( x\right) .
\end{equation*}%
The last equality gives us that $\widehat{u}_{1}\in H^{1+\varepsilon }\left(
\mathbb{R}
\right) $, and consequently $u_{1}\in H_{0}^{1+\varepsilon }\left(
0,1\right) $. Hence, by the definition of $\phi $, we get%
\begin{equation*}
\phi \left( v\right) =\left\langle u_{0},v\right\rangle +\left\langle
u_{1},v^{\prime }\right\rangle
\end{equation*}%
\begin{equation*}
=\left\langle u_{0},v\right\rangle -\left\langle u_{1}^{\prime
},v\right\rangle =\left\langle u_{0}-u_{1}^{\prime },v\right\rangle =0,
\end{equation*}%
for every $v\in X$.

Now, to complete the proof of the lemma, it is sufficient to prove (A.1) for
$f\in \mathcal{D}\left[ 0,1\right] $. Let $f\in \mathcal{D}\left[ 0,1\right]
$, $\alpha \in \mathcal{D[}0,1)$ and $\alpha (x)=1$, for $x\in \mathcal{[}0,%
\frac{1}{2}]$. Define $\widetilde{f}\left( x\right) :=\left\{
\begin{array}{l}
\alpha \left( x\right) f\left( x\right) ,\text{ }x\in \lbrack 0,1), \\
0,\text{ }x>1,%
\end{array}%
\right. $ and $\Phi \left( x\right) :=\left\{
\begin{array}{l}
\widetilde{f}\left( x\right) ,\text{ }x>0, \\
\widetilde{f}\left( -x\right) ,\text{ }x\leq 0.%
\end{array}%
\right. $ It is easy to verify that $\Phi \in H^{1}\left(
\mathbb{R}
\right) $ and%
\begin{equation}
\left\Vert \Phi \right\Vert _{H^{-\varepsilon }\left(
\mathbb{R}
\right) }+\left\Vert \Phi ^{\prime }\right\Vert _{H^{-\varepsilon }\left(
\mathbb{R}
\right) }\leq c_{1}\left( \left\Vert f\right\Vert _{H^{-\varepsilon }\left(
0,1\right) }+\left\Vert f^{\prime }\right\Vert _{H^{-\varepsilon }\left(
0,1\right) }\right) .  \tag{A.3}
\end{equation}%
On the other hand, by using Fourier transformation, one can show that
\begin{equation}
\left\Vert \Phi \right\Vert _{H^{1-\varepsilon }\left(
\mathbb{R}
\right) }\leq c_{2}\left( \left\Vert \Phi \right\Vert _{H^{-\varepsilon
}\left(
\mathbb{R}
\right) }+\left\Vert \Phi ^{\prime }\right\Vert _{H^{-\varepsilon }\left(
\mathbb{R}
\right) }\right) .  \tag{A.4}
\end{equation}%
Taking into account the continuous embedding $H^{1-\varepsilon }\left(
\mathbb{R}
\right) \hookrightarrow C_{b}\left(
\mathbb{R}
\right) $, by (A.3) and (A.4), we obtain%
\begin{equation*}
\left\Vert \alpha f\right\Vert _{C\left[ 0,1\right] }\leq c_{3}\left(
\left\Vert f\right\Vert _{H^{-\varepsilon }\left( 0,1\right) }+\left\Vert
f^{\prime }\right\Vert _{H^{-\varepsilon }\left( 0,1\right) }\right) .
\end{equation*}%
Similarly, one can prove%
\begin{equation*}
\left\Vert \left( 1-\alpha \right) f\right\Vert _{C\left[ 0,1\right] }\leq
c_{4}\left( \left\Vert f\right\Vert _{H^{-\varepsilon }\left( 0,1\right)
}+\left\Vert f^{\prime }\right\Vert _{H^{-\varepsilon }\left( 0,1\right)
}\right) ,
\end{equation*}%
which, together with the previous inequality, yields (A.1).
\end{proof}

\begin{lemma}
Let $f\in L^{\infty }\left( 0,1\right) $ and $p>2$. Then, for every $u\in
W^{1,1}\left( 0,1\right) $ such that%
\begin{equation}
\left\vert u^{\prime }+\left\vert u\right\vert ^{p-2}u\right\vert \leq f%
\text{ \ \ a.e. in }\left( 0,1\right) ,  \tag{A.5}
\end{equation}%
the following estimates hold:%
\begin{equation}
\left\Vert u\right\Vert _{C\left[ s_{0},1\right] }\leq \left( \frac{p}{p-2}%
+\left\Vert f\right\Vert _{L^{\infty }\left( 0,1\right) }\right) ^{\frac{1}{%
p-2}},  \tag{A.6}
\end{equation}%
and%
\begin{equation}
\left\Vert u^{\prime }\right\Vert _{L^{\infty }\left( s_{0},1\right) }\leq
\left( \frac{p}{p-2}+\left\Vert f\right\Vert _{L^{\infty }\left( 0,1\right)
}\right) ^{\frac{p-1}{p-2}}+\left\Vert f\right\Vert _{L^{\infty }\left(
0,1\right) },  \tag{A.7}
\end{equation}%
where $s_{0}=1-\frac{\left( \frac{p}{p-2}\right) ^{\frac{1}{p-2}}-1}{\left(
\frac{p}{p-2}\right) ^{\frac{p-1}{p-2}}+\left\Vert f\right\Vert _{L^{\infty
}\left( 0,1\right) }}$.
\end{lemma}

\begin{proof}
Since (A.7) immediately follows from (A.5) and (A.6), we will only prove
(A.6). We consider the following cases:

\textit{Case 1. }Assume that\textit{\ }$\left\vert u\left( 1\right)
\right\vert \leq 1$. Then, by the continuity of $u$, the set $E:=\left\{
t\in \left[ 0,1\right) :\left\vert u\left( s\right) \right\vert <\right. $%
\newline
$\left. \left( \frac{p}{p-2}\right) ^{\frac{1}{p-2}}\text{, for }s\in \left[
t,1\right] \right\} $ is nonempty. Let $\alpha =\inf E$. If $\alpha =0$,
then from the definition of $E$, we get (A.6). If $\alpha \in \left(
0,1\right) $, then again by the definition of $E$ and the continuity of $u$,
we have%
\begin{equation*}
\left\vert u\left( \alpha \right) \right\vert =\left( \frac{p}{p-2}\right) ^{%
\frac{1}{p-2}},
\end{equation*}%
\begin{equation}
\left\vert u\left( t\right) \right\vert <\left( \frac{p}{p-2}\right) ^{\frac{%
1}{p-2}},\text{ \ }\forall t\in \left( \alpha ,1\right] ,  \tag{A.8}
\end{equation}%
and consequently, by (A.5),%
\begin{equation*}
\left\vert u^{\prime }\right\vert <\left( \frac{p}{p-2}\right) ^{\frac{p-1}{%
p-2}}+\left\Vert f\right\Vert _{L^{\infty }\left( 0,1\right) },\text{ \ a.e.
in }(\alpha ,1).
\end{equation*}%
Considering the last inequality, we find%
\begin{equation*}
\left( \frac{p}{p-2}\right) ^{\frac{1}{p-2}}=\left\vert u\left( \alpha
\right) \right\vert \leq \left\vert u\left( 1\right) \right\vert
+\int\limits_{\alpha }^{1}\left\vert u^{\prime }\left( t\right) \right\vert
dt
\end{equation*}%
\begin{equation*}
\leq 1+\left( \left( \frac{p}{p-2}\right) ^{\frac{p-1}{p-2}}+\left\Vert
f\right\Vert _{L^{\infty }\left( 0,1\right) }\right) \left( 1-\alpha \right)
,
\end{equation*}%
which yields $\alpha \leq s_{0}$. Thus, from (A.8), we obtain (A.6).

\textit{Case 2. }Assume that\textit{\ }$\left\vert u\left( 1\right)
\right\vert >1$. Then the set $\widetilde{E}:=\left\{ t\in \left[ 0,1\right)
:\left\vert u\left( s\right) \right\vert >1,\text{ for }s\in \lbrack
t,1]\right\} $ is nonempty. Let $\beta =\inf \widetilde{E}$. Multiplying
both sides of (A.5) by $\left( p-2\right) \left\vert u\left( t\right)
\right\vert ^{p-3}e^{(p-2)\int\limits_{0}^{t}\left\vert u\left( \tau \right)
\right\vert ^{p-2}d\tau }$ and integrating the obtained inequality over $%
\left( s,T\right) $, we get%
\begin{equation*}
\left\vert u\left( T\right) \right\vert ^{p-2}\leq e^{-\left( p-2\right)
\int\limits_{s}^{T}\left\vert u\left( t\right) \right\vert
^{p-2}dt}\left\vert u\left( s\right) \right\vert ^{p-2}
\end{equation*}%
\begin{equation*}
+\left( p-2\right) \left\Vert f\right\Vert _{L^{\infty }\left( 0,1\right)
}\int\limits_{s}^{T}e^{-\left( p-2\right) \int\limits_{t}^{T}\left\vert
u\left( \tau \right) \right\vert ^{p-2}d\tau }\left\vert u\left( t\right)
\right\vert ^{p-3}dt\text{, \ \ }\beta \leq s\leq T\leq 1.
\end{equation*}%
By the definition of $\beta $, we have%
\begin{equation*}
\left( p-2\right) \int\limits_{s}^{T}e^{-\left( p-2\right)
\int\limits_{t}^{T}\left\vert u\left( \tau \right) \right\vert ^{p-2}d\tau
}\left\vert u\left( t\right) \right\vert ^{p-3}dt
\end{equation*}%
\begin{equation*}
\leq \left( p-2\right) \int\limits_{s}^{T}e^{-\left( p-2\right)
\int\limits_{t}^{T}\left\vert u\left( \tau \right) \right\vert ^{p-2}d\tau
}\left\vert u\left( t\right) \right\vert ^{p-2}dt
\end{equation*}%
\begin{equation*}
=\int\limits_{s}^{T}\frac{d}{dt}e^{-\left( p-2\right)
\int\limits_{t}^{T}\left\vert u\left( \tau \right) \right\vert ^{p-2}d\tau
}dt\leq 1,
\end{equation*}%
for every $s\in \lbrack \beta ,T]$. Hence, by the last two inequalities, we
find%
\begin{equation}
\left\vert u\left( T\right) \right\vert ^{p-2}\leq e^{-\left( p-2\right)
\int\limits_{s}^{T}\left\vert u\left( t\right) \right\vert
^{p-2}dt}\left\vert u\left( s\right) \right\vert ^{p-2}+\left\Vert
f\right\Vert _{L^{\infty }\left( 0,1\right) },\text{ \ \ }\beta \leq s\leq
T\leq 1.  \tag{A.9}
\end{equation}%
Now, if $\beta \in \left( 0,s_{0}\right] $, then by the continuity of $u$,
we have $u\left( \beta \right) =1$. So, choosing $s=\beta $ in (A.9), we
obtain (A.6). If $\beta \in \left( s_{0},1\right) $, then again choosing $%
s=\beta $ in (A.9) and taking into account that $u\left( \beta \right) =1$,
we get%
\begin{equation*}
\left\vert u\left( T\right) \right\vert ^{p-2}\leq 1+\left\Vert f\right\Vert
_{L^{\infty }\left( 0,1\right) },\ \ \ \forall T\in \left[ \beta ,1\right] .
\end{equation*}%
Since $u\left( \beta \right) =1$, taking $t=\beta $ instead of $t=1$ and
applying the procedure of Case 1, one can show that%
\begin{equation*}
\left\vert u\left( t\right) \right\vert ^{p-2}\leq \frac{p}{p-2},
\end{equation*}%
for every $t\in \left[ s_{0},\beta \right] .$ So, by the last two
inequalities, we again obtain (A.6). If $\beta =0$, then integrating (A.9)
over $\left[ 0,T\right] $ with respect to $s$, we get%
\begin{equation*}
T\left\vert u\left( T\right) \right\vert ^{p-2}\leq \frac{1}{p-2}%
+T\left\Vert f\right\Vert _{L^{\infty }\left( 0,1\right) },
\end{equation*}%
and consequently%
\begin{equation*}
\left\vert u\left( T\right) \right\vert ^{p-2}\leq \frac{1}{s_{0}(p-2)}%
+\left\Vert f\right\Vert _{L^{\infty }\left( 0,1\right) }<\left\Vert
f\right\Vert _{L^{\infty }\left( 0,1\right) }
\end{equation*}%
\begin{equation*}
+\frac{1}{\left( 1-\frac{p-2}{p}\right) (p-2)}<\frac{p}{p-2}+\left\Vert
f\right\Vert _{L^{\infty }\left( 0,1\right) },
\end{equation*}%
for every $T\in \lbrack s_{0},1]$. The last inequality gives us (A.6).
\end{proof}

\end{document}